\documentclass[12pt]{article}
\usepackage{amsthm, amsmath, amssymb, enumerate}
\usepackage{fullpage}
\usepackage[colorlinks=true]{hyperref}

\theoremstyle{plain}
\newtheorem{thm}{Theorem}[section]

\newtheorem{cor}[thm]{Corollary}

\newtheorem{lem}[thm]{Lemma}

\newtheorem{prop}[thm]{Proposition}

\theoremstyle{remark}

\numberwithin{equation}{section}

\title{Rationality of N\'eron--Tate height over function fields}
\author{Chengyuan Yang}
\date{\today}
\begin{document}
	\maketitle
	\tableofcontents
	\section{Introduction}
	Let $K$ be the function field of a geometrically integral normal projective curve $S$ over $k$, $A$ an abelian variety with an ample symmetric line bundle $L$ on $A$. We prove that the N\'eron--Tate height $\hat{h}_L(Z)$ is a rational number for any closed subvariety $Z$ of $A$. Height functions characterize the complexity of subvarieties. Originally, height is defined on points of $\mathbb{P}^n$. The classical height is the map $h:\mathbb{P}^n(\overline{K})\to\mathbb{R}$ given by
	\begin{equation*}
		h(x_0, \ldots, x_n)=\frac{1}{[K':K]}\sum\limits_{v\in M_{K'}}{\rm log} {\rm max}\{|x_0|_v, \ldots|x_n|_v\}
	\end{equation*}
	for $K'$ a finite extension of $K$ that contains a homogeneous coordinate of the point $(x_0,\ldots, x_n)$. The valuation $|\cdot|_v$ is normalized by $|t|_v=e^{-[k(v):k]{\rm ord}_v(t)}$. Fix a projective variety $X$ over $K$ and an ample line bundle $L$ on $X$. Then, for some positive integer $e$, $L^{\otimes e}$ is very ample, which gives an embedding $\varphi: X\to\mathbb{P}^n$. Define the \emph{Weil height} by $h_L(x)=\frac1eh(\varphi(x))$. Note that Weil height depends on the choice of $\varphi$. Different choices of $\varphi$ change the Weil height by a bounded function. 
	
	When $X=A$ is an abelian variety, we are able to define a height that does not depend on such choice. For simplicity, we assume that $L$ is symmetric (i.e. $L\cong[-1]^*L$). The \textit{N\'eron--Tate height}, or \textit{canonical height}, is defined by Tate's limit argument
	$$\hat{h}_L(x)=\lim\limits_{n\to+\infty}4^{-n}h_L(2^nx),$$
	where $h_L$ can be taken as any Weil height of $L$. The canonical height does not depend on the choice of $h_L$. N\'eron--Tate height is defined by taking limit, so it is surprising that it is still a rational number by the N\'eron function \cite[Th\'eor\`eme 4]{neron1965quasi} or the Moret-Bailly model \cite[III.3.3]{moret1985pinceaux}. In this paper, we generalize this result to heights of subvarieties of arbitrary dimension. 
	\begin{thm}\label{MT}
		Let $A$ be an abelian variety over a function field $K$, $L$ a symmetric ample line bundle on $A$, and $Z$ a closed subvariety of $A$. Then the N\'eron--Tate height $\hat{h}_L(Z)$ is a rational number. 
	\end{thm}
	Robin de Jong \cite[Corollary 1.2]{de2018neron} has shown that N\'eron--Tate heights of certain tautological cycles on jacobians are contained in a $3$-dimensional $\mathbb{Q}$-linear subspace of $\mathbb{R}$ in number field case. De Jong's proof can be transferred to function field case, showing that these heights are rational numbers. De Jong and Farbod Shokrieh \cite[Section 1.6]{de2022faltings} have also shown that, in function field case, the N\'eron-Tate height of a symmetric theta divisor is a rational number when the abelian variety is principally polarized. We prove that in function field case, N\'eron--Tate heights are always rational numbers. This explains why the coefficients in \cite[Theorem 1.1]{de2018neron} are rational numbers. 
	
	The N\'eron--Tate height for subvarieties is defined as follows. A projective model $(\mathcal{A},\mathcal{L})$ of $(A,L)$ is an integral projective scheme $\mathcal{A}$ over $S$ with generic fiber $A$ and a line bundle $\mathcal{L}$ on $\mathcal{A}$ such that $\mathcal{L}_K\cong L$. For any closed subvariety $Z$ of $A$, the naive height with respect to the projective model $(\mathcal{A},\mathcal{L})$ is the intersection number
	\begin{equation*}
		h_{\mathcal{L}}(Z):=\frac{\mathcal{L}^{{\rm dim}(Z)+1}\cdot\mathcal{Z}}{(1+{\rm dim}(Z))){\rm deg}_L(Z)},
	\end{equation*}
	where $\mathcal{Z}$ is the Zariski closure of $Z$ in $\mathcal{A}$, and ${\rm deg}_L(Z)$ is the top self-intersection number of $L$ on $Z$. Let $[n]:A\to A$ be the morphism of multiplication by $n\in\mathbb{Z}$. The \textit{N\'eron--Tate height} of $Z$ is the limit 
	\begin{equation*}
		\hat{h}_L(Z)=\lim\limits_{n\to+\infty}4^{-n}h_{\mathcal{L}}([2^n]_*Z).
	\end{equation*}	
	This definition is independent of the choice of the projective model. 

	The proof is by induction on the dimension of the subvariety. Fix a rigidification on $L$. We describe the N\'eron--Tate height using adelic intersection numbers \cite[3.1]{zhang1995small}
	\begin{equation*}
		\hat{h}_L(Z)=\frac{\langle\hat{L}^{{\rm dim}(Z)+1}|Z\rangle}{({\rm dim}(Z)+1){\rm deg}_L(Z)},
	\end{equation*}
	where $\hat{L}$ is an adelic line bundle extending $L$ depending on the rigidification. By the result of Chambert-Loir and Thuillier \cite[Th{\'e}or{\`e}me 4.1]{chambert2009mesures}, for a non-zero rational section $s$ of $\hat{L}$, we have the induction formula 
	\begin{equation*}
		\langle\hat{L}^{{\rm dim}(Z)+1}|Z\rangle=\langle\hat{L}^{{\rm dim}(Z)}|Z\cap{\rm div}(s)\rangle+\sum\limits_{v\in S}\int_{Z_v^{\rm an}}-{\rm log}\|s\|_{\hat{L}}c_1(\hat{L})_v^{{\rm dim}(Z)}\delta_{Z_v}.
	\end{equation*}
	Note that the sum in the right-hand side is a finite sum. Thus, we need to prove that local integrals are a rational numbers. 
	
	The Chambert-Loir measure $c_1(\hat{L})_v^{{\rm dim}(Z)}\delta_{Z_v}$ is supported on the skeleton of $Z_v$, which is a $\mathbb{Q}$-piecewise linear space by \cite[Theorem 5.1.1]{berkovich2004smooth} and \cite[Theorem 6.12]{gubler2010non}. Walter Gubler \cite[Theorem 1.1]{gubler2010non} has shown that the measure $c_1(\hat{L})_v^{{\rm dim}(Z)}\delta_{Z_v}$ is a linear combination of Lebesgue measures of rational simplices compatible with the piecewise linear structure. It is easy to check that the linear combination has $\mathbb{Q}$-coefficients.
	
	By \cite[Theorem 5.1.1]{berkovich2004smooth}, all meromorphic functions are compatible with the piecewise linear space in the sense that for a meromorphic function $f$ on $Z_v^{\rm an}$, the morphism $Z_v^{\rm an}\to\mathbb{R}: x\to-{\rm log}|f(x)|$ is a $\mathbb{Q}$-piecewise linear function on the skeleton. Thus, for a model metrized line bundle $\overline{M}$ on $Z_v$ with a rational section $f$, the integral of $-{\rm log}\|f\|_{\overline{M}}$ on a rational simplex in the skeleton is a rational number. However, in general, the metrized line bundle $\hat{L}|_{Z_v}$ does not come from a model. 
	
	To treat the canonical metric, we use the functional equation for any non-zero rational section $s$ of $L$
	\begin{equation*}
		-4{\rm log}\|s(x)\|_{\hat{L}}+{\rm log}\|s(2x)\|_{\hat{L}}=-{\rm log}|h(x)|,
	\end{equation*}
	where $h:=s^{\otimes 4}\otimes[2]^*s^{\otimes -1}$ is a rational function on $A$ and $x$ is a point in $A^{\rm an}$. By \cite[Proposition 11.1.4]{lang2013fundamentals}, the functional equation determines a unique metric. We construct a metric following the functional equation by a model metric and a tropical term. This generalizes the result in \cite[Proposition 9.1]{de2022faltings} where de Jong and Shokrieh treat the case when $L$ determines a principal polarization. 
	
	In section \ref{MLB}, we review the general theory about metrics on line bundles, including adelic line bundles and Chambert-Loir measure. In section \ref{A}, we review the uniformization theory of abelian varieties. The proof of Theorem \ref{MT} is in section \ref{PF}. \\
	
	\noindent\textbf{Acknowledgments.} I would like to express my deep gratitude to my advisor, Professor Xinyi Yuan, for his constant encouragement and guidance throughout my study. I am grateful to Walter Gubler and Robin de Jong, since the article is greatly inspired by their works and communications with them. I would also like to thank my classmates Ruoyi Guo, Lai Shang, and She Yang for some useful discussions.

	\section{Metrized line bundles}\label{MLB}
	In this section, We review the theory of adelic line bundles and metrized line bundles introduced in Zhang \cite{zhang1995small} and the induction formula proved in Chambert-Loir--Thuillier \cite{chambert2009mesures}. Note that we use the terminology introduced in Yuan--Zhang \cite{yuan2021adelic}. 
	\subsection{Berkovich spaces}
	Let $K$ be a discrete non-archimedean valued field, $R$ the valuation ring of $K$. For a variety $X$ over $K$, we construct an analytic space $X^{\rm an}$. 
	
	For the case $X={\rm Spec}(A)$, the set of points in $X^{\rm an}$ is defined as $\mathcal{M}(A)$, which is the set of all multiplicative semi-norms extending the norm on $K$. For any $x\in\mathcal{M}(A)$, $f\in A$, let $|f(x)|=|f|_x$ where $|\cdot|_x$ is the semi-norm associated to $x$. The topology is the weakest one such that $|f|:\mathcal{M}(A)\to \mathbb{R}$ is continuous for all $f\in A$. 
	
	For the general case, choose a finite open affine cover ${\rm Spec}(A_i)$ of $X$. Then $X^{\rm an}$ is defined by gluing together $\mathcal{M}(A_i)$ in a natural way. A subset in $X^{\rm an}$ is open if and only if its intersection with $\mathcal{M}(A_i)$ is open in $\mathcal{M}(A_i)$ for all $i$. 
	
	Write $S={\rm Spec}(R)$. A projective model of $X$ over $S$ is a projective, integral scheme over $S$ whose generic fiber is $X$. Suppose we have a projective model $\mathcal{X}$ of $X$ over $S$. Then there is a reduction map ${\rm red}: X^{\rm an}\to \mathcal{X}$ defined as follows. For any $x\in X^{\rm an}$, let $H_x$ be the residue field of $x$, $R_x$ the valuation ring of $H_x$. By the valuative criterion of properness, the $K$-morphism ${\rm Spec}(H_x)\to X$ extends to a unique $S$-morphism ${\rm Spec}(R_x)\to \mathcal{X}$. Let ${\rm red}(x)$ be the image of the closed point of ${\rm Spec}(R_x)$. Note that the image of ${\rm red}$ is contained in the special fiber of $\mathcal{X}$. 
	\subsection{Metrics}\label{M}
	Let $K$ be a discrete non-archimedean valued field, $R$ the valuation ring of $K$. For a projective variety $X$ over $K$, let $X^{\rm an}$ be the Berkovich space associate to $X$. A \emph{metrized line bundle} on $X$ is a line bundle $L$ on $X$ with a metric $\|\cdot\|$ on $X^{\rm an}$ such that for any open subset $U$ of $X^{\rm an}$ with a section $s\in\Gamma(L,U)$, $\|s\|$ gives a continuous function $U\to\mathbb{R}_{\geq 0}$, $\|fs\|=\|s\|\cdot|f|$ for all analytic functions $f$ on $U$, and $\|s(x)\|$ does not vanish if $s$ does not vanish at $x$.
	
	Let $\mathcal{L}$ be a line bundle on $\mathcal{X}$ with generic fiber $L$. For a point $x\in X^{\rm an}$, consider an open neighborhood $\mathcal{U}$ of ${\rm red}(x)$ in $\mathcal{X}$ such that $\mathcal{L}$ is trivialized on $\mathcal{U}$ and write $U=\mathcal{U}_K$. Choose a trivialization $t\in\mathcal{L}(\mathcal{U})$. For any $s\in L(U)$, define $\|s(x)\|_\mathcal{L}=|f(x)|$ where $f$ is the unique rational function on $U$ such that $f\cdot t=s$. Note that different choice of $t$ differs by a unit in $R_x$, and thus $\|\cdot\|_\mathcal{L}$ does not depend on the choice of $t$. The metric $\|\cdot\|_\mathcal{L}$ is called the model metric on $L$ associated to the model $(\mathcal{X},\mathcal{L})$. A metric $\|\cdot\|$ on $L$ is called a \textit{model metric} if it comes from a projective model $(\mathcal{X},\mathcal{L})$ of $L^{\otimes m}$ for some positive integer $m$. A line bundle is called \textit{nef} if it has non-negative degree on any curve in the special fiber, and a model metric is called \textit{nef} if it comes from a nef model. A metric is called \textit{nef} if it is a uniform limit of nef model metrics. A metrized line bundle $\overline{L}$ is called \textit{integrable} if it is the difference of two nef metrized line bundle. 
	
	Let $A$ be an abelian variety over $K$. For any rigidified symmetric ample line bundle $L$ on $A$, there is a canonical isomorphism $\varphi_m: [m]^*L\stackrel{\sim}{\to} L^{\otimes m^2}$. By \cite[Theorem 2.2]{zhang1995small}, there is a unique metrized line bundle $\overline{L}$ extending $L$ such that $[2]^*\overline{L}\stackrel{\sim}{\to}\overline{L}^{\otimes4}$ is an isometry. This metric is called the \textit{canonical metric} of $L$. Write $\|\cdot\|_L$ for the canonical metric of $L$. When $A$ has good reduction, the canonical metric of $L$ is a model metric. In general, the canonical metric of an ample line bundle is always integrable.
	\subsection{Adelic line bundles}\label{ALB}
	Let $K$ be the function field of a projective normal curve $S$ over $k$. For each $v\in S$, we choose the non-archimedean norm on $K$ defined as follows. For any $f\in K$, let $|f|_v=e^{-[k_v:k]{\rm ord}_v(f)}$ where $k_v$ is the residue field of $v$. This choice of valuation satisfies the product formula. Let $K_v$ be the completion of $K$ under the ultrametric $|\cdot|_v$. 
	
	Let $X$ be a projective variety over $K$. An adelic line bundle is a line bundle $L$ and a collection of continuous metrics $\|\cdot\|_v$ on $X_v:=X\times_{{\rm Spec}(K)}{\rm Spec}(K_v)$ for all $v\in S$ such that there is an open subset $U$ of $S$ with a projective model $\mathcal{X}$ of $X$ over $U$ and a line bundle $\mathcal{L}$ on $\mathcal{X}$ satisfying for all $v\in U$, $\|\cdot\|_v$ comes from the model $(\mathcal{X}_v, \mathcal{L}_v):=(\mathcal{X}\times_U{\rm Spec}(K_v), \mathcal{L}\times_U{\rm Spec}(X_v))$. A model adelic line bundle is an adelic line bundle $\overline{L}$ such that $\overline{L}^m$ comes from a projective model for some positive integer $m$. 
	
	Let $\|\cdot\|_n:n\in\mathbb{Z}_+$ be a sequence of adelic metrics on $L$, we say that $\|\cdot\|_n$ converge to $\|\cdot\|_n$ if for all but finite many $v\in S$, $\|\cdot\|_{n,v}=\|\cdot\|_v$ for all $n$, and for other $v\in S$, $\|\cdot\|_{n,v}$ converge uniformly to $\|\cdot\|_v$. A model adelic line bundle is called \textit{nef} if it comes from a model $(\mathcal{X}, \mathcal{L})$ with $\mathcal{L}$ nef (i.e. has non-negative degree on any curve in special fibers). An adelic line bundle is called \textit{nef} if it is the limit of a sequence of nef model adelic line bundles, called \textit{integrable} if it is the difference of two nef adelic line bundles. 
	
	By \cite[Theorem 1.4]{zhang1995small}, for integrable adelic line bundles $\overline{L}_1,\overline{L}_2, \ldots, \overline{L}_n$ where $n={\rm dim}(X)+1$, there is a intersection number $\langle\overline{L}_1, \overline{L}_2, \ldots, \overline{L}_n\rangle\in\mathbb{R}$ defined as follows. When all adelic line bundles are nef, suppose $\overline{L}_i$ is the limit of nef model adelic line bundles $\lim\limits_{j\to\infty}\overline{\mathcal{L}}_{i,j}^{\frac{1}{e_{i,j}}}$ where $\overline{\mathcal{L}}_{i,j}$ comes from a projective model $(\mathcal{X}_{i,j}, \mathcal{L}_{i,j})$, and $e_{i,j}$ is a positive integer. The generic fiber of $\mathcal{X}_{1,j}\times_S\mathcal{X}_{2,j}\times_S\ldots\times_S\mathcal{X}_{n,j}$ is $X\times_{{\rm Spec}(K)}X\times_{{\rm Spec}(K)}\ldots\times_{{\rm Spec}(K)}X$. Let $\Delta$ be the image of $X$ under the diagonal map to the generic fiber, and let $\mathcal{X}_j$ be the Zariski closure of $\Delta$ in $\mathcal{X}_{1,j}\times_S\mathcal{X}_{2,j}\times_S\ldots\times_S\mathcal{X}_{n,j}$. We have maps $p_{i,j}: \mathcal{X}_j\to\mathcal{X}_{i,j}$ which is the composition of the immersion and the projection map. Then, $\overline{\mathcal{L}}_{i,j}^{e_{i,j}}$ comes from the projective model $(\mathcal{X}_j, p_{i,j}^*\mathcal{L}_{i,j})$. The intersection number $\langle\overline{L}_1, \overline{L}_2, \ldots, \overline{L}_n\rangle$ is the limit $\lim\limits_{j\to\infty}\frac{\langle p_{1,j}^*\overline{\mathcal{L}}_{1,j}, p_{2,j}^*\overline{\mathcal{L}}_{2,j}, \ldots, p_{n,j}^*\overline{\mathcal{L}}_{n,j} \rangle}{e_{1,j}e_{2,j}\cdots e_{n,j}}$ where the numerator is the algebraic intersection number on $\mathcal{X}_j$. Extending by linearity and we get the case for integrable adelic line bundles. 
	
	Let $A$ be an abelian variety over $K$, with a rigidified symmetric ample line bundle $L$ on it. Let $\hat{L}$ be the line bundle $L$ endowed with canonical metrics at all places. Choose a flat projective model $(\mathcal{A},\mathcal{L})$ of $(A,L^{\otimes m})$ over $S$ such that $\mathcal{L}$ is relatively ample. There is an open subset $U\in S$ such that the morphism $[2]$ and the isomorphism $[2]^*L\stackrel{\sim}{\to} L^{\otimes4}$ extend to the morphism $[2]_U:\mathcal{A}_U\to\mathcal{A}_U$ and the isomorphism $[2]^*_U\mathcal(L)_U\stackrel{\sim}{\to}\mathcal(L)_U^{\otimes 4}$. Let $f_n: \mathcal{X}_n\to\mathcal{X}$ be the normalization of the composition of morphisms
	\begin{equation*}
		\mathcal{X}_U\stackrel{[2]^n_U}{\longrightarrow}\mathcal{X}_U\longrightarrow\mathcal{X}.
	\end{equation*}
	Let $\|\cdot\|_{\mathcal{L}_n}$ be the metric on $L$ associated to the projective model $(\mathcal{X}_n,f_n^*\mathcal{L})$, then $\|\cdot\|_{\mathcal{L}_n}$ is a converge uniformly to $\|\cdot\|_{\hat{L}}$. Since $f_n^*\mathcal{L}$ is relatively ample, and hence nef, $\hat{L}$ is an integrable adelic line bundle. This also implies that the canonical metric in the local case is always integrable. 
	
	For a closed subvariety $Z$ of dimension $d$ in $A$, we redefine the \textit{N\'eron--Tate height} by the intersection number
	\begin{equation*}
		\hat{h}_L(Z)=\frac{\langle\hat{L},\ldots,\hat{L}|Z\rangle}{(d+1){\rm deg}_L(Z)}. 
	\end{equation*}
	As discussed in \cite[Section 1.1]{xie2022geometric}, this definition is the same as the one defined in the introduction. 
	
	\subsection{Chambert-Loir measure}
	We use terminologies and assumptions in section \ref{ALB}. Fix a place $v\in S$, $K_v$ is a discrete non-archimedean valued field, and $X_v$ is a projective variety. In \cite{chambert2006mesures}, Chambert-Loir constructed a measure which follows the induction formula over $K_v$. By \cite[Th\'eor\`eme 4.1]{chambert2009mesures}, for any integrable metrized line bundles $\overline{L}_1, \ldots, \overline{L}_n$ on $X_v$ there is a regular Borel measure $c_1(\overline{L}_1)c_1(\overline{L}_2)\ldots c_1(\overline{L}_n)$ on $X$, such that for any integrable metrized line bundle $\overline{L}_{n+1}$ and rational sections $s_1, s_2, \ldots, s_{n+1}$ of $\overline{L}_1, \overline{L}_2, \ldots, \overline{L}_{n+1}$ respectively, we have the induction formula
	\begin{equation*}
		\langle\widehat{\rm div}(s_1)\ldots\widehat{\rm div}(s_{d+1})\rangle=\langle\widehat{\rm div}(s_1)\ldots\widehat{\rm div}(s_d) \rangle|{\rm div}(s_{d+1})-\int_{X^{\rm an}}\|s_{d+1}\|_{\overline{L}_{d+1}} c_1(\overline{L}_1) \ldots c_1(\overline{L}_d).
	\end{equation*} 
	
	Chambert-Loir measure gives the induction formula for the global case as well. For instance, for any integrable adelic line bundles $\overline{L}_1, \overline{L}_2, \ldots, \overline{L}_{n+1}$, and a rational section $s$ of $\overline{L}_{n+1}$, with $Z:={\rm div}(s)$, the equation 
	\begin{equation}\label{Induction}
		\langle\overline{L}_1, \overline{L}_2, \ldots, 	\overline{L}_{n+1}\rangle=\langle\overline{L}_1, \overline{L}_2, \ldots, \overline{L}_n|Z\rangle+\sum\limits_{v\in S}\int_{X_v^{\rm an}}-{\rm log}\|s\|_{L_{n+1}}c_1(\overline{L}_1)_v\ldots c_1(\overline{L}_n)_v
	\end{equation}
	holds. 
	
	As in section \ref{ALB}, there is a projective model $\mathcal{X}$ with line bundles $\mathcal{L}_1, \mathcal{L}_2, \ldots, \mathcal{L}_{n+1}$ and positive integers $e_1, e_2, \ldots, e_{n+1}$ such that $\overline{L}_i^{e_i}$ differs from $\mathcal{L}_i$ on only finite many places. Since the induction formula for algebraic intersection number is a finite sum, the right-hand side of equation (\ref{Induction}) is a finite sum. 
	\section{Abelian varieties}\label{A}
	The goal of this section is to study the rationality of the canonical metric and the canonical measure. Our method is the theory of Raynaud extensions of abelian varieties developed in Bosch--L\"utkebohmert \cite{bosch1991degenerating}.
	\subsection{Raynaud extensions}\label{Ray}
	Let $K$ be a discrete non-archimedean valued field, $R$ the valuation ring of $K$. Let $A$ be an abelian variety over $K$ with good or split semi-stable reduction over $R$. First, we discuss the Raynaud extension of $A$. 
	
	Write $S=\text{Spec}(R)$. Let $p: E^{\rm an}\to A^{\rm an}$ be the universal cover, with a point $0\in E^{\rm an}$ such that $p(0)$ is the identity of $A^{\rm an}$. As discussed in \cite{bosch1991degenerating}, there is a unique abelian scheme $\mathcal{B}$ over $S$ with generic fiber $B$, a unique analytic group structure on $E^{\rm an}$ with identity $0$, a split torus $T=\mathbb{G}_m^n$ over $K$, and a lattice $M$ of rank $n$ such that we have exact sequences
	\begin{equation}\label{RE}
		0\longrightarrow T\longrightarrow E\stackrel{q}{\longrightarrow} B\longrightarrow 0,
	\end{equation}
	\begin{equation*}
		0\longrightarrow M\longrightarrow E^{\rm an}\longrightarrow A^{\rm an}\longrightarrow 0.
	\end{equation*}
	
	Let $f_i:\mathbb{G}_m^n\to\mathbb{G}_m$ be the projection to the $i$th factor. Let $E_i$ be the pushout of $E$ by $f_i$, then we have an exact sequence
	\begin{equation*}
		0\longrightarrow\mathbb{G}_m\longrightarrow E_i\longrightarrow B\longrightarrow 0.
	\end{equation*}
	Thus, $E_i$ is a $\mathbb{G}_m$-torsor over $B$, and can be identified as a rigidified line bundle $H_i$ on $B$. The map $E\to E_i$ can be regarded as a trivialization of $q^*E_i$, and we write $e_i$ for this trivialization. Let $\|\cdot\|_{H_i}$ be the canonical metric of $H_i$, which is a model metric since $B$ has good reduction. Now we get a continuous surjective map 
	$${\rm val}: E\longrightarrow\mathbb{R}^n, x\longrightarrow (\log\|e_i(x)\|_{H_i})_{1\leq i\leq n}.$$
	By \cite[Theorem 1.2]{bosch1991degenerating}, $\rm val$ gives a homeomorphism from $M$ to a complete lattice $\Gamma\subset \mathbb{R}^n$. Thus, val induces a continuous surjective map
	$$\overline{\rm val}: A^{\rm an}\longrightarrow \mathbb{R}^n/\Gamma.$$
	
	Following \cite[Example 7.2]{gubler2010non}, there is a map $\iota: \mathbb{R}^n
	/\Gamma\to A^{\rm an}$ such that $\overline{\rm val}\circ\iota$ is the identity on $\mathbb{R}^n/\Gamma$ and that $\overline{\tau}:=\iota\circ\overline{\rm val}$ is a deformation retraction. The image of $\overline{\tau}$, denoted by $S(A)$, is called the canonical skeleton of $A$. As described in \cite[Section 6.5]{berkovich2012spectral}, $\overline{\tau}$ extends to a unique deformation retraction $\tau: E^{\rm an}\to E^{\rm an}$ with image $S(E)=p^{-1}(S(A))$. 
	
	Now we describe line bundles using the uniformization. Let $L$ be an ample symmetric rigidified line bundle on $A$. 
	
	For any $t\in M$, let $\alpha_t$ be the automorphism $x\to x-t$ on $E$. Then $p^*L^{\rm an}$ is equipped with the action $\alpha_t$ by the uniformization $A^{\rm an}=E^{\rm an}/M$. By \cite[Proposition 4.4]{bosch1991degenerating}, there is a rigidified line bundle $H$ on $B$ satisfying $p^*L^{\rm an}=q^*(H^{\rm an})$. We further assume that $L$ can be written as the form $N\otimes[-1]^*N$ for some line bundle $N$ on $A$, then $H$ can be assumed to be symmetric. Then, $q^*(\overline{H}^{\rm an})$ determines a metric on $p^*L^{\rm an}$ where $\overline{H}$ is the canonical metrized line bundle of $H$. Write $\|\cdot\|_{q^*H}$ to be this metric. 
	
	There is a map $z_t: E\to \mathbb{R}$ such that for any $f\in H^0(E^{\rm an}, p^*L^{\rm an})$, $-{\rm log}\|\alpha_t(f)(x+t)\|_{q^*H}=z_t(x)-{\rm log}\|f(x)\|_{q^*H}$. By \cite[Proposition 4.9]{bosch1991degenerating}, $z_t(x)=z_t(\tau(x)))$. A \textit{theta function} is a global section of $p^*L^{\rm an}$ that descends to a global section of $L^{\rm an}$. If $f$ is a theta function, then for all $t\in M$there exists $c(t)\in K$ depending only on $L$ and $t$ such that $\alpha_t(f)=c(t)f$. Thus, we obtain a functional equation (cf. \cite[Propersition 3.13(2)]{foster2018non})
	\begin{equation}\label{T}
		-{\rm log}\|f(x)\|_{q^*H}=-{\rm log}\|f(x+t)\|_{q^*H}-z_t(x)-{\rm log}|c(t)|.
	\end{equation}
	\subsection{Canonical metrics}\label{CMet}
	We use terminologies and assumptions in section \ref{Ray}. Let $s$ be a non-zero rational section of $L$. Through the isomorphism $\varphi_2: [2]^*L\stackrel{\sim}{\to} L^{\otimes 4}$ given by the rigidification of $L$, $h:=s^{\otimes 4} \otimes [2]^*s^{\otimes -1}$ can be identified as a rational function on $A$. Since $\varphi_2$ is an isometry for the canonical metric on $L$, the equality
	\begin{equation}\label{NF}
		-4{\rm log}\|s(x)\|_L+{\rm log}\|s(2x)\|_L=-{\rm log}|h(x)|
	\end{equation}
	holds for all $x\in A^{\rm an}$. 
	
	Write $D={\rm div}_L(s)$. The restriction of $-{\rm log}\|s(x)\|_L$ on $A\setminus {\rm supp}(D)$ is a \textit{Weil function} associated with 
	$D$ in the sense of \cite[pg.255]{lang2013fundamentals}. By \cite[Proposition 11.1.4]{lang2013fundamentals}, the functional equation (\ref{NF}) for ${\rm log}\|s\|_L$ determines a unique Weil function on $A\setminus {\rm supp}(D)$. 
	
	\begin{prop}\label{S}
		Let $s$ be a non-zero rational section of $L$, $f$ the theta function of $s$. Let $y\in E^{\rm an}$, and write $x=p(y)$. Then, the equality 
		\begin{equation*}
			-{\rm log}\|s(x)\|_L+{\rm log}\|s(\overline{\tau}(x))\|_L=-{\rm log}\|f(y)\|_{q^*H}+{\rm log}\|f(\tau(y))\|_{q^*H}
		\end{equation*}
		holds whenever $s$ does not vanish at $x$ {\rm (cf. \cite[Proposition 9.1]{de2022faltings})}. 
	\end{prop}
	\begin{proof}
		For any $X\in A^{\rm an}$ choose a lift $y\in E^{\rm an}$. Let $\Lambda(x)=-{\rm log}\|s(\overline{\tau}(x))\|_L-{\rm log}\|f(y)\|_{q^*H}+{\rm log}\|f(\tau(y))\|_{q^*H}$. First, we show that $\Lambda(x)$ does not depend on the choice of $y$. Suppose $y$ and $y'$ are both lifts of $x$ in $E$, then $t:=y'-y$ is in the lattice $M$. By \ref{T}, we have
		\begin{align*}
			-{\rm log}\|f(y)\|_{q^*H}+{\rm log}\|f(y+t)\|_{q^*H}&=-z_t(y)-{\rm log}|c(t)|\\
			&=-z_t(\tau(y))-{\rm log}|c(t)|\\
			&=-{\rm log}\|f(\tau(y))\|_{q^*H}+{\rm log}\|f(\tau(y+t))\|_{q^*H}.
		\end{align*}
		
		Let $D={\rm div}_L(s)$, $h=s^{\otimes 4} \otimes [2]^*s^{\otimes -1}$. Then, $\Lambda(x)$ is a Weil function associated to $D$. By the continuity of $\Lambda$ on the locus where $s$ does not vanish and the density of $A\setminus{\rm supp}(4D-[2]^*D)$ in $A^{\rm an}$, we need only to prove that equation (\ref{NF}) holds for $\Lambda$ on $A\setminus{\rm supp}(4D-[2]^*D)$. That is, to prove that 
		\begin{equation}\label{Lambda}
			4\Lambda(x)-\Lambda(2x)=-{\rm log}|h(x)|
		\end{equation}
		holds for all $x\in A\setminus{\rm supp}(4D-[2]^*D)$.
		
		Let $g(z)=f(z)^4f(2z)^{-1}$, since $\|\cdot\|_H$ is the canonical metric on $B$, we have
		\begin{equation*}
			-4{\rm log}\|f(z)\|_{q^*H}+{\rm log}\|f(2z)\|_{q^*H}=-{\rm log}|g(z)|.
		\end{equation*}
		Since $g$ and $p^*h$ have same zeros and poles, $g^{-1}p^*h$ and $(p^*h)^{-1}g$ are both holomorphic functions on $E^{\rm an}$. By \cite[Theorem 5.2]{berkovich1999smooth}, we conclude that 
		\begin{equation*}
			|g^{-1}(z)(p^*h)(z)|=|g^{-1}(\tau(z))(p^*h)(\tau(z))|. 
		\end{equation*}
		Thus, we have
		\begin{flalign*}
			4\Lambda(x)-\Lambda(2x)=&\ -4{\rm log}\|f(y)\|_{q^*H}+{\rm log}\|f(2y)\|_{q^*H}\\
		 	&\ +4{\rm log}\|f(\tau(y))\|_{q^*H}-{\rm log}\|f(\tau(2y))\|_{q^*H}\\
			&\ -4{\rm log}\|s(\overline{\tau}(x))\|_L+{\rm log}\|s(\overline{\tau}(2x))\|_L\\
			=&\ -{\rm log}|g(y)|+{\rm log}|g(\tau(y))|-{\rm log}|(p^*h)(\tau(y))|\\
			=&\ -{\rm log}|h(x)|.
		\end{flalign*}
	\end{proof}
	
	Next, we study the behavior of $\|s\|_L$ on the canonical skeleton $S(A)$. We assume that the image of the map $K\to \mathbb{R}: x\to{\rm log}|x|$ lies in $\mathbb{Q}$. Recall the skeleton of $S(A)$ is homeomorphic to $\mathbb{R}^n/\Gamma$. Let $\Delta$ be a rational simplex in $\mathbb{R}^n$(i.e. the convex hull of at most $n+1$ rational point of general position), the measure $\delta_{\overline{\Delta}}$ is the pushforward of the Lebesgue measure on $\Delta$ to $\mathbb{R}^n/\Gamma$. 
	
	All piecewise linear in the following of this section is piecewise $|K^*|_{\mathbb{Z}}$-linear in the sense of \cite[Section 1]{berkovich2004smooth}. Following \cite[Section 5]{berkovich2004smooth}, the skeleton of $A^{\rm an}$ is a piecewise linear space. By \cite[Theorem 5.1.1]{berkovich2004smooth}, for any meromorphic function $f$ on $A^{\rm an})$, the morphism $|f|:X^{\rm an}\to\mathbb{R}:x\to|f(x)|$ is a piecewise linear function while restricted to the skeleton $S(A)$. If we assume that ${\rm log}|K^*|\subset\mathbb{Q}$, then the integral of ${\rm log}|f|$ on any rational simplex is a rational number. Thus, we have the following proposition. 
	\begin{prop}\label{MetricSkelton}
		Assume that  ${\rm log}|K^*|\subset\mathbb{Q}$, the integral 
		\begin{equation}\label{Int}
			\int_{x\in\mathbb{R}^n/\Gamma}-{\rm log}\|s(\iota(x))\|_L\delta_{\overline{\Delta}}
		\end{equation}
		is a rational number. 
	\end{prop}
	\begin{proof}
		Write $d={\rm dim}(\Delta)+1$. Let $H_1, H_2, \ldots, H_d$ be the $(d-2)$-dimensional affine subspaces in $\mathbb{R}^n$ that contain a face of $\Delta$, $H_0$ the $(d-1)$-dimensional affine subspace that contains $\Delta$. For each $0\leq i\leq d$, let $2^tH_i$ be the affine subspace $\{x\in \mathbb{R}^n|2^{-t}x\in H_i\}$, and $2^t\overline{H}_i$ be the image of $2^tH_i$ in $\mathbb{R}^n/\Gamma$. Then, there exist distinct integers $m_i$, $n_i$ such that $2^{m_i} \overline{H}_i = 2^{n_i} \overline{H}_i$. Thus, the set $V:=\{2^t\overline{H}_i|t\in\mathbb{Z}_+, 1\leq i\leq d\}$ is a finite set. Let $\Omega$ be the set of all connected component of $\bigcup\limits_{t=0}^{\infty}2^t\overline{H}_0\setminus\bigcup\limits_{H'\in V}H'$. For each $\Delta'\in\Omega$ with $\Delta'\subset\overline{\Delta}$, the number of fibers of the map $\Delta\to \mathbb{R}^n/\Gamma$ is constant on $\Delta'$. For each $\Delta'\in\Omega$, let $\delta_{\Delta'}$ be the pushforward of the Lebesgue measure of a lift of $\Delta'$ in $\mathbb{R}^n$, then $\delta_{\overline{\Delta}}$ is a linear combination of integral coefficients of $\delta_{\Delta'}$ for $\Delta'\in\Omega$. Also, there exist non-negative rational numbers $r_{\Delta_1,\Delta_2}$ ($\Delta_1,\Delta_2\in\Omega$) such that for $\Delta_0\in\Omega$, $[2]_*\delta_{\Delta_1}=\sum\limits_{\Delta_2\in\Omega}r_{\Delta_1, \Delta_2} \delta_{\Delta_2}$. 
				
		Let $F:\Omega\to\mathbb{R}$ be the map $\Delta'\to \int_{x\in\mathbb{R}^n/\Gamma}-{\rm log}\|s(\iota(x))\|_L\delta_{\Delta'}$. By the functional equation (\ref{NF}), 
		\begin{equation*}
			4\int_{x\in\mathbb{R}^n/\Gamma}-{\rm log}\|s(\iota(x))\|_L\delta_{\Delta'}+\int_{x\in\mathbb{R}^n/\Gamma}{\rm log}\|s(\iota(x))\|_L[2]_*\delta_{\Delta'}=\int_{x\in\mathbb{R}^n/\Gamma}-{\rm log}|f(\iota(x))|\delta_{\Delta'}.
		\end{equation*}
		The right-hand side in the above equation is rational since $-{\rm log}|f|$ is piecewise linear on the skeleton. Thus, we obtain that $4F-TF$ is a rational vector in $\mathbb{R}^{\Omega}$. 
		
		Let $I$ be the identity map on $\mathbb{R}^{\Omega}$, we prove that $4I-T$ is invertible. For any vector $v:=(v_{\Delta_1})$ in $\mathbb{R}^{\Omega}$, let $m_+(v)=\sum\limits_{\Delta_1\in\Omega}m(\delta_{\Delta_1}){\rm max}\{v_{\Delta_1},0\}$ where $m(\delta_{\Delta'})$ is the total mass of $\delta_{\Delta'}$. Note that $m(\delta_{\Delta_1})=\sum\limits_{\Delta_2\in\Omega}r_{\Delta_1, \Delta_2} m(\delta_{\Delta_2})$. Since all coefficient in $T$ are positive, $m_+(Tv)\leq m_+(v)$. This shows that if $4v=Tv$, then $m_+(v)=0$. Also, we have $m_+(-v)=0$. We conclude that $v=0$. Thus, $4I-T$ is invertible. Since $4I-T$ has rational coefficients, so does $(4I-T)^{-1}$. Therefore, $F$ has rational coordinates. The integral (\ref{Int}), being a linear combination of integral coefficient of coordinates of $F$, is a rational number. 
	\end{proof}

	\subsection{Canonical measures}\label{CMS}
	Let $K$ be a discrete non-archimedean valued field, $A$ an abelian variety over $K$, $Z$ a closed subscheme of dimension $d$ on $A$. Let $L$ be a symmetric ample line bundle on $A$, then the canonical metrized line bundle $\overline{L}$ is integrable. For an ample symmetric rigidified line bundle $L$ on $A$, the canonical metrized line bundle $\overline{L}$ is a uniform limit of ample line bundles, and is hence integrable. 
	
	When $A$ has good reduction, $\overline{L}$ has a model metric. By \cite[D\'efinition 2.4]{chambert2006mesures}, the Chambert-Loir measure $c_1(\overline{L})^d\delta_Z:=c_1(\overline{L}|_Z)^d$ supports on finite many points such that the measure of each point a rational number. 
	
	In the general case, the canonical measure is calculated in Gubler \cite{gubler2010non}. Following \cite[Theorem 1.1]{gubler2010non}, the Chambert-Loir measure $c_1(\overline{L})^d\delta_Z$ can be described as follows: there exist rational simplices $\Delta_1, \ldots, \Delta_N$ in $\mathbb{R}^n$ and real numbers $r_1, r_2,\ldots, r_N$ in $\mathbb{R}$ such that 
	\begin{equation}\label{CM}
		\overline{\rm val}_*(c_1(\hat{L})^d\delta_Z) =\sum\limits_{i=1}^Nr_i\delta_{\overline{\Delta}_i}
	\end{equation}
	where $\overline{\Delta}_i$ is the pushforward of the Lebesgue measure on $\Delta_i$ to $\mathbb{R}^n/\Gamma$. 
	
	First, we recall the process in \cite{gubler2010non} for calculating the canonical measure. By de Jong's alteration theorem \cite[Theorem 6.5]{de1996smoothness}, and by possibly take finite field extensions, there is a strictly semi-stable $S$-scheme $\mathcal{X}'$ with generic fiber $X'$, and a dominant proper morphism $f:X'\to X$ such that $f$ is finite on an open subscheme of $X$. \cite[Theorem 5.2]{berkovich1999smooth} constructed the skeleton of $X'$. Here, we describe the canonical measure $c_1(f^*\bar{L})^d$ on $X'$. 
	
	Let ${\rm str}(\mathcal{X}'_s)$ be the strata of the special fiber of $\mathcal{X}'$ (the definition of strata is in \cite[1.2]{gubler2010non}). By \cite[Remark 5.16]{gubler2010non}, there is an open covering $U_S$ ($S\in {\rm str}(\mathcal{X}'_s)$)of $X'$ with the following properties. For each $S\in{\rm str}(\mathcal{X}'_s)$, $U_S$ admits a semi-stable model, and the skeleton of $U_S$, denoted by $\Delta_S$, is a rational simplex contained in the skeleton of $\mathcal{X'}$. Also, the map $U_S\to A$ lifts to a map $\tilde{f}_S: U_S\to E$ to the uniformization of $A$. By \cite[Proposition 5.11] {gubler2010non}, ${\rm val}\circ\tilde{f}_S$ gives an affine map from $\Delta_S$ to $\mathbb{R}^n$. The simplex $\Delta_S$ is called non-degenerate if ${\rm val}\circ\tilde{f}_S$ is injective on $\Delta_S$. Write $\delta_{{\Delta}_S}$ for the Lebesgue measure on $\Delta_S$. As described in \cite[Theorem 6.7]{gubler2010non}, the Chambert-Loir measure $c_1(f^*\bar{L})^d$ is a linear combination of $\delta_{\Delta_S}$. We check that the coefficient of this linear combination is rational. 
	\begin{lem}\label{CLA}
		There are rational numbers $t_S$ ($S\in{\rm str}(\mathcal{X}')_s$) such that $t_S=0$ if $\Delta_S$ non-degenerate and that 
		\begin{equation*}
			c_1(f^*\bar{L})^d=\sum\limits_{S\in {\rm str}(\mathcal{X}'_s)}t_S\delta_{{\Delta}_S}. 
		\end{equation*}
	\end{lem}
	\begin{proof}
		By \cite[Theorem 6.7]{gubler2010non}, 
		\begin{equation}\label{Measure}
			c_1(f^*\bar{L})^d=\sum\limits_{S\in {\rm str}(\mathcal{X}'_s)\atop\Delta_S {\rm\text{ is non-degenerate}} }\frac{d!}{(d-e)!}{\rm deg}_\mathcal{H}(\overline{S})\frac{{\rm vol}(\Lambda_S^L)}{{\rm vol}(\Lambda_S)}\delta_{{\Delta}_S}
		\end{equation}
		where $e$ is the dimension of $S$, $\Lambda_S^L$ and $\Lambda_S$ are two complete rational lattices on $\mathbb{R}^{{\rm dim}\Delta_S}$ defined in \cite[6.6]{gubler2010non}, and ${\rm deg}_\mathcal{H}(\overline{S})$ is the degree of a line bundle on $\overline{S}$ defined in \cite[5.17, 6.6]{gubler2010non}. Since all terms in the right-hand side of (\ref{Measure}) are rational numbers, the coefficients $t_S$ are rational numbers. 
	\end{proof}
	
	We now prove that $r_i$ are actually rational numbers. 
	\begin{cor}
		The coefficients $r_1, r_2, \ldots, r_N$ in the equation (\ref{CM}) are all rational numbers. 
	\end{cor}
	\begin{proof}
		By lemma \ref{CLA} and the projection formula \cite[Propri\'et\'es 2.8]{chambert2006mesures}, 
		\begin{equation*}
			\overline{\rm val}_*(c_1(\hat{L})^d\delta_Z)=\sum\limits_{S\in {\rm str}(\mathcal{X}'_s)}{\rm deg}(f)t_S({\rm \overline{val}}\circ f)_*\delta_{{\Delta}_S}. 
		\end{equation*}
		Now fix a $S\in{\rm str}(\mathcal{X}'_s)$ with $\Delta_S$ non-degenerate. Since ${\rm val}\circ \tilde{f}_S$ is an injective affine map on $\Delta_S$, the image of $\Delta_S$ is a simplex in $\mathbb{R}^n$, denoted by $\Delta'_S$. Then $({\rm \overline{val}}\circ f)_*\delta_{{\Delta}_S}$ is a rational number times $\delta_{\overline{\Delta}_S}$. The lemma follows. 
		
	\end{proof}

	\section{Proof of the main theorem}\label{PF}
	Let $K$ be the function field of a normal projective curve $S$ over $k$, $A$ an abelian variety over $K$, $L$ an ample symmetric rigidified line bundle on $A$. Then, we have the following theorem. 
	\begin{thm}{\rm (Theorem \ref{MT})}
		For any closed subvariety $Z$ in $A$, the N\'eron--Tate height $\hat{h}_L(Z)$ is a rational number. 
	\end{thm}
	\begin{proof}
	Write $d={\rm dim}(Z)$. By taking field extension of $K$, we assume that $A$ has good or split semi-stable reductions at all places. Possibly replace $L$ by $L^{\otimes2}$ so that the line bundle there exist a line bundle $N$ with $L=N\otimes N^{\otimes 1}$. Fix a rational section $s$ of $L$ such that the support of ${\rm div}(s)$ does not contain $Z$. By the induction formula (\ref{Induction}), we need to prove that 
	\begin{equation*}
		\sum\limits_{v\in S}\int_{Z_v^{\rm an}}-{\rm log}\|s\|_{\hat{L}} c_1(\hat{L})_v^d \delta_{Z_v}
	\end{equation*}
	is a rational number. Recall the right-hand side of the above equation is a finite sum. 
	
	
	Fix $v\in S$. Let $g$ be the theta function of $s$ on the uniformization of $A^{\rm an}_v$. Consider the semistable alteration $X'$ of $Z_v$. Fix $S\in{\rm str(\mathcal{X}')}$. Let $U_S$, $\Delta_S$, $\tilde{f}_S$ as in section \ref{CMS}. Then, $\tilde{f}_S^*g$ restricts to a piecewise linear function on $\Delta_S$. When $\Delta_S$ is non-degenerate, by proposition \ref{S}, the equation
	\begin{align*}
		\int_{{X'}^{\rm an}}-{\rm log}\|f^*s(x)\|_{\hat{L}} \delta_{\Delta_S}=\ &\int_{X'^{\rm an}}-{\rm log} \|s(\overline{\tau}(f(x))\|_{\hat{L}} \delta_{{\Delta}_S}\\
		&+\int_{X'^{\rm an}}-{\rm log}\|{\tilde{f}}_S^*g(x)\|_{q^*H} \delta_{{\Delta}_S}\\
		&-\int_{X'^{\rm an}}-{\rm log}\|g(\tau(\tilde{f}_S(x)))\|_{q^*H} \delta_{{\Delta}_S}
	\end{align*}
	holds. The first term in the right-hand side is a rational number by proposition \ref{MetricSkelton}. Note that the metric $\|\cdot\|_{q^*H}$ is a model metric, and hence the map $x\to-{\rm log}\|g(x)\|_{q^*H}$ is piecewise linear on the skeleton. The second and third term in the right-hand side are rational numbers since they are integrals of piecewise linear functions on rational simplices. Thus, the right-hand side is also a rational number. 
	
	On the other hand, by lemma \ref{CLA}, the measure $c_1(\hat{L})_v^d \delta_{Z_v}$
	is a linear combination of rational coefficient of measures $f_*(\delta_{{\Delta}_S})$. The theorem follows. 
	\end{proof}
	\bibliographystyle{alpha}
	\bibliography{References}

\begin{thebibliography}{FRSS18}

\bibitem[Ber99]{berkovich1999smooth}
Vladimir~G Berkovich.
\newblock Smooth p-adic analytic spaces are locally contractible.
\newblock {\em Inventiones mathematicae}, 137(1):1, 1999.

\bibitem[Ber04]{berkovich2004smooth}
Vladimir~G Berkovich.
\newblock Smooth p-adic analytic spaces are locally contractible. ii.
\newblock {\em Geometric aspects of Dwork theory}, 1:293--370, 2004.

\bibitem[Ber12]{berkovich2012spectral}
Vladimir~G Berkovich.
\newblock {\em Spectral theory and analytic geometry over non-Archimedean
  fields}.
\newblock Number~33. American Mathematical Soc., 2012.

\bibitem[BL91]{bosch1991degenerating}
Siegfried Bosch and Werner L{\"u}tkebohmert.
\newblock Degenerating abelian varieties.
\newblock {\em Topology}, 30(4):653--698, 1991.

\bibitem[CL06]{chambert2006mesures}
Antoine Chambert-Loir.
\newblock Mesures et {\'e}quidistribution sur les espaces de berkovich.
\newblock {\em Journal f{\"u}r die reine und angewandte Mathematik},
  595:215--235, 2006.

\bibitem[CLT09]{chambert2009mesures}
Antoine Chambert-Loir and Amaury Thuillier.
\newblock Mesures de mahler et {\'e}quidistribution logarithmique.
\newblock In {\em Annales de l'Institut Fourier}, volume~59, pages 977--1014,
  2009.

\bibitem[dJ96]{de1996smoothness}
Aise~Johan de~Jong.
\newblock Smoothness, semi-stability and alterations.
\newblock {\em Publications Math{\'e}matiques de l'IH{\'E}S}, 83:51--93, 1996.

\bibitem[dJ18]{de2018neron}
Robin de~Jong.
\newblock N{\'e}ron--tate heights of cycles on jacobians.
\newblock {\em Journal of Algebraic Geometry}, 27:339--381, 2018.

\bibitem[dJS22]{de2022faltings}
Robin de~Jong and Farbod Shokrieh.
\newblock Faltings height and n{\'e}ron--tate height of a theta divisor.
\newblock {\em Compositio Mathematica}, 158(1):1--32, 2022.

\bibitem[FRSS18]{foster2018non}
Tyler Foster, Joseph Rabinoff, Farbod Shokrieh, and Alejandro Soto.
\newblock Non-archimedean and tropical theta functions.
\newblock {\em Mathematische Annalen}, 372:891--914, 2018.

\bibitem[Gub10]{gubler2010non}
Walter Gubler.
\newblock Non-archimedean canonical measures on abelian varieties.
\newblock {\em Compositio Mathematica}, 146(3):683--730, 2010.

\bibitem[Lan13]{lang2013fundamentals}
Serge Lang.
\newblock {\em Fundamentals of Diophantine geometry}.
\newblock Springer Science \& Business Media, 2013.

\bibitem[MB85]{moret1985pinceaux}
Laurent Moret-Bailly.
\newblock Pinceaux de vari{\'e}t{\'e}s ab{\'e}liennes.
\newblock {\em Ast{\'e}risque}, 129, 1985.

\bibitem[N{\'e}r65]{neron1965quasi}
Andr{\'e} N{\'e}ron.
\newblock Quasi-fonctions et hauteurs sur les vari{\'e}t{\'e}s ab{\'e}liennes.
\newblock {\em Annals of Mathematics}, pages 249--331, 1965.

\bibitem[XY22]{xie2022geometric}
Junyi Xie and Xinyi Yuan.
\newblock Geometric bogomolov conjecture in arbitrary characteristics.
\newblock {\em Inventiones mathematicae}, 229(2):607--637, 2022.

\bibitem[YZ21]{yuan2021adelic}
Xinyi Yuan and Shou-Wu Zhang.
\newblock Adelic line bundles over quasi-projective varieties.
\newblock {\em arXiv preprint arXiv:2105.13587}, 2021.

\bibitem[Zha95]{zhang1995small}
Shouwu Zhang.
\newblock Small points and adelic metrics.
\newblock {\em Journal of Algebraic Geometry}, 4(2):281--300, 1995.

\end{thebibliography}
\end{document}